\documentclass[12pt,reqno]{amsart}
\usepackage{amsthm, amssymb, amstext, amscd, amsfonts, amsxtra, latexsym, amsmath, comment, tikz,
	mathrsfs, mathtools, esint, stmaryrd, cancel,scalerel,hyperref,booktabs}
\usepackage{geometry}
\theoremstyle{plain}
\newtheorem{theorem}{Theorem}[section]
\newtheorem{corollary}[theorem]{Corollary}

\newtheorem{lemma}[theorem]{Lemma}
\newtheorem{proposition}[theorem]{Proposition}
\theoremstyle{definition}

\theoremstyle{remark}
\newtheorem*{remark}{Remark}

\numberwithin{equation}{section}

\newcommand{\N}{\mathbb N}
\newcommand{\Z}{\mathbb Z}

\makeatletter
\newcommand{\vast}{\bBigg@{4}}
\newcommand{\Vast}{\bBigg@{5}}
\makeatother

\renewcommand{\pmod}[1]{\  \,  \left( \rm{mod} \,  #1 \right)}

\newlength{\parenheight}
\newlength{\parendepth}
\newlength{\parendrop}

\newcommand{\paren}[4]{%
	\settoheight{\parenheight}{\(#4 #2\)}%
	\settodepth{\parendepth}{\(#4 #2\)}
	\addtolength{\parendepth}{.5ex}
	\addtolength{\parenheight}{-.5ex}
	\addtolength{\parenheight}{\parendepth}
	\addtolength{\parendepth}{-.5\parenheight}
	\setlength{\parendrop}{-.5\parenheight}
	\addtolength{\parendrop}{.5ex}
	\raisebox{-\parendepth}{\(#4
		\left#1%
		\rule[\parendrop]{0pt}{\parenheight}%
		\right.\)}
	#2
	\raisebox{-\parendepth}{\(#4
		\left.%
		\rule[\parendrop]{0pt}{\parenheight}%
		\right#3\)}
}

\def\myleft#1#2\myright#3{%
	\mathchoice{%
		\paren{#1}{#2}{#3}{\displaystyle}%
	}{%
	\paren{#1}{#2}{#3}{\textstyle}%
}{%
\paren{#1}{#2}{#3}{\scriptstyle}%
}{%
\paren{#1}{#2}{#3}{\scriptscriptstyle}%
}%
}

\renewenvironment{proof}[1][Proof]{\begin{trivlist} \item[\hskip \labelsep {\bfseries #1:}]}{\qed\end{trivlist}}

\begin{document}
\allowdisplaybreaks
\author{Chris Jennings-Shaffer and Antun Milas}

\address{ Department of Mathematics, University of Denver, 2390 S. York St. Denver, CO 80208.}

\email{christopher.jennings-shaffer@du.edu}

\address{Department of Mathematics and Statistics,  SUNY-Albany, 1400 Washington Avenue  Albany, NY 12222.}

\email{amilas@albany.edu}

\title{On $q$-series identities for false theta series }
\thispagestyle{empty} \vspace{.5cm}
\maketitle

\begin{abstract} We prove several infinite families of $q$-series identities for false theta functions 
and related series. These identities are motivated by considerations of characters of modules of vertex 
operator superalgebras and of quantum dilogarithms. We also obtain 
closely related modular identities of the G\"ollnitz-Gordon-Andrews type. As a byproduct of our identities, we establish several identities for the Rogers dilogarithm function coming from multi $q$-hypergeometric series with ``double poles''.
\end{abstract}


\section{Introduction}
False theta functions are similar to theta functions except for an alteration of the signs in some of the series terms, which prevents them from being modular forms. They appear in a variety of contexts including
modular and quantum modular forms \cite{BM0, FOR, LZ}, meromorphic Jacobi forms \cite{BFOR}, quantum knot 
invariants \cite{BeirneOsburn1, GL,Hajij, HL,KO} (e.g. tails of colored Jones polynomials), 
Witten-Reshetikhin-Turaev invariants of Seifert 3-manifolds \cite{BHL, Hi1,Hi2}, and many more. More recently, they have appeared in the context of ``homological blocks'' of 3d-theories $T[M]$ with the gauge group $SU(2)$, where $M$ is a plumbed $3$-manifold \cite{GPPV} (see also \cite{BMM,CCFGH}). 
There are also higher-rank generalizations of false theta functions introduced and studied by K. Bringmann, T. Creutzig,
and the second author \cite{BM,CM} (see also \cite{BKM}).


The literature on $q$-series identities connecting partial and false theta functions with $q$-hypergeometric series is extensive (see \cite{Warnaar2} and references therein). One of the most elegant identities in this context was given by 
Ramanujan  \cite[p. 18]{Berndt}\begin{equation} \label{Ramanujan}
 \sum_{n \in \mathbb{Z}} {\rm sgn}(n) q^{2n^2+n}=(q)_\infty \sum_{n \geq 0} \frac{q^{n^2+n}}{(q)_n^2},
\end{equation}
where ${\rm sgn}(n):=1$ for $n \geq 0$ and $-1$ otherwise (as usual, $(a)_n:=\prod_{i=1}^n (1-aq^{i-1})$).
The $q$-series on the left-hand side is an example of Rogers' false theta function.

Ramanujan's identity can be generalized to all positive integral characteristics.
\begin{theorem}[\cite{BM,Hajij,Warnaar4}] \label{ak} For $k \in \N$, we have 
\begin{equation} \label{MFI}
\frac{q^{-\frac{k^2}{4(k+1)}}}{(q)_\infty} {\sum_{n \in \mathbb{Z}} {\rm sgn}(n) q^{(k+1)(n+\frac{k}{2(k+1)})^2}}
=\sum_{n_1,n_2,\dotsc,n_{k} \geq 0} \frac{q^{N_1^2+N_2^2+\cdots+N_{k}^2+N_1+N_2+\cdots + N_{k}}}
	{(q)_{n_{k}}^2 (q)_{n_1} (q)_{n_2} \cdots (q)_{n_{k-1}}},
\end{equation}
where $N_i=\sum_{j \geq i} n_j$.
\end{theorem}
This theorem can be approached in several different ways; it follows from 
the analytic form of the Andrews-Gordon identities combined with standard $q$-summation techniques \cite{BM0}, or using relations 
for the tail of colored Jones polynomials of $(2,2k)$ torus knots \cite{Hajij}, or using Bailey's Lemma \cite{Warnaar2, Warnaar4}. From the vertex algebra perspective, this particular form is interesting because the left-hand side in (\ref{MFI}) is the character of the $(1,k+1)$-singlet vertex algebra \cite{BM}; this explains a somewhat unusual shift in the exponent and the Euler factor appearing 
in the denominator (cf. (\ref{Ramanujan})). For additional aspects of partial theta and theta functions, we refer the reader to the 
recent review article by O. Warnaar \cite{Warnaar3} 
and references therein; see also \cite{FF,Warnaar1} for related identities in the context of conformal field theory.

More recently, in \cite{AM2,AM1, Sidoli}, a different type of false theta function was (re)discovered in connection with 
the $N=1$ super-singlet vertex superalgebras. Characters of irreducible modules of the super-singlet
are similar to the expression in (\ref{MFI}), but they contain a  ``fermionic'' product $(-q^{\frac12})_\infty$
and their characteristic is half-integral ($k \in \mathbb{N}$, $\epsilon \in \{0,\frac12\}$):
\begin{equation} \label{half-char}
 \frac{(-q^{\frac12+\epsilon})_\infty}{(q)_\infty} \sum_{n \in \mathbb{Z}} {\rm sgn}(n)q^{(k+\frac12)(n+a)^2},
\end{equation}
for some $a \in \mathbb{Q}$.
Regularized modular properties of these expressions have been thoroughly investigated in \cite{AMS,Sidoli}.

Characters of modules for the $N=1$ super-singlet algebra combine into characters of modules for certain ``logarithmic'' 
vertex algebras called $N=1$ super-triplets \cite{AM1}. Their irreducible characters take the following form (see \cite{AM2,AM1})
\begin{equation} \label{half-char-full}
\frac{(-q^{\frac12+\epsilon})_\infty}{(q)_\infty}  \sum_{n \in \mathbb{Z}} (2n+1) q^{(k+\frac12)(n+a)^2}.
\end{equation}
This series, up to a multiplicative $q$-power, is known to be modular. More precisely, the theta-like series in 
\eqref{half-char-full}
is a sum of modular forms of weight $\frac12$ and $\frac32$. 

Although the origin of false theta functions is not completely understood, their primary source is usually attributed to meromorphic Jacobi forms \cite[Chapter 11]{BFOR}. 
The simplest instance of this arises in the following well-known Fourier expansion \cite{A2, TanneryMolk1} ($|q|<|\zeta|^2<1$):
\begin{equation}\label{EqInverseJacobi}		
\frac{1}{\left(\zeta q^{\frac12},\zeta^{-1} q^{\frac12}\right)_\infty}
=
	\frac{1}{(q)^2_\infty}   
	\sum_{n_1} \sum_{n_2\geq |n_1|}  
	(-1)^{n_1+n_2} q^{\frac{ n_2(n_2+1)}{2}-\frac{n_1^2}{2}}\zeta^{n_1}
= 
	\sum_{\substack{n_1 \in \mathbb{Z} \\ n_2\geq 0}} \frac{q^{\frac{|n_1|}2+n_2}}{(q)_{n_2} (q)_{|n_1|+n_2}} \zeta^{n_1},
\end{equation}
where $(a_1,...,a_k)_m:=(a_1)_m \cdots (a_k)_m$.
As is plainly seen, extracting the constant term with respect to $\zeta$ yields 
another elegant expression for the Rogers false theta function (with $k=1$),
\begin{equation} \label{euler-p2}
\frac{1}{(q)^2_\infty} \sum_{n \in \mathbb{Z}} {\rm sgn}(n) q^{2n^2+n}=\sum_{n \geq 0} \frac{q^{n}}{(q)_n^2}.
\end{equation}
Compared to formula (\ref{Ramanujan}), this identity comes with an extra Euler factor and 
has no quadratic term in the exponent of the $q$-hypergeometric term.
	
The main purpose of this paper is to motivate and prove:
\begin{itemize}
\item[(I)] An analog of (\ref{MFI}) for certain $q$-series as in (\ref{half-char}). In particular, this would give 
a combinatorial form for irreducibile characters of $N=1$ super-singlet modules.
\item[(II)] A G\"ollnitz-Andrews-Gordon-type expression for the $q$-series in (\ref{half-char}). Related identities 
in a non-super context were studied in \cite{Warnaar1}.
\item[(III)] A family of $q$-series identities for false theta series coming from Fourier expansions
 of multi-variable Jacobi forms, in parallel to the $k=1$ case seen in (\ref{euler-p2}).
\end{itemize}
Let us briefly outline our solutions to these problems and the contents of the paper.

In Section 2, we gather several facts on Bailey pairs and Bailey chains used in the paper. 
These powerful methods are applied in Section 3 to prove identities for ``shifted'' 
false theta functions. In  Proposition \ref{PropGeneralShifted1/2}, we present identities for $\epsilon=0$ whereas Proposition \ref{PropGeneralShifted1} pertains to $\epsilon=\frac12$.
As a special case, we obtain $q$-hypergeometric expressions for (\ref{half-char}).

In Section 4, we first use Lemma \ref{shift-help}, to write (\ref{half-char-full}) as an infinite sum of shifted false theta series. 
Then we combine this with identities from Section 3, and obtain a single multi-hypergeometric series with a parity condition. 
Our main results here are Theorems  \ref{NS-char-log}, \ref{log-Ramond},
and \ref{log-Ramond2}, with Theorem \ref{NS-char-log} 
corresponding to $\epsilon=0$ and 
Theorems \ref{log-Ramond} and \ref{log-Ramond2}
corresponding to $\epsilon=\frac12$.

In Section 5, we switch our focus to $q$-series identities generalizing (\ref{euler-p2}); see Theorem \ref{euler-false-p}. 
We first express the relevant $q$-hypergometric series in a form convenient for use of Bailey chains from Section 2
(see Lemma \ref{LemmaMultiSum1}).
Interestingly, this methods allows us to prove more than we hoped for; we also obtain a family of modular identities 
in Theorem \ref{euler-modular-p} that intertwine with the ``false'' identities in Theorem \ref{euler-false-p}.
We should point out that Theorem \ref{euler-modular-p} was recently conjectured in \cite{CS} in connection to Argyres-Douglas theories in physics.


In Section 6, using well-established methods \cite{Kirillov,VZ, Zagier}, we 
prove several identities for the Rogers dilogarithm function (see Propositions  
\ref{dilog-log}, \ref{PropositionMoreDilogarithms1}, and \ref{PropositionMoreDilogarithms2}). 
We finish in Section 7 with a few pointers for future research.

{\bf Acknowledgments.} We thank C. Nazaroglu for pointing out \cite{CS}. We also thank O. Warnaar, J. Lovejoy,  R. Osburn and K. Bringmann for 
informative discussions and suggestions. Lastly, we thank the anonymous referees
for their various corrections and improvements to an earlier version of
this manuscript.


\section{Bailey pairs, the Bailey chain, and the Bailey Lattice}

We recall a pair of sequences $(\alpha_n,\beta_n)$ is called a Bailey pair
relative to $(a,q)$ if
\begin{gather*}
\beta_n = \sum_{j=0}^n\frac{\alpha_j}{(q)_{n-j}(aq)_{n+j}}.
\end{gather*}
Bailey's lemma is an identity relating rather general $q$-series involving
Bailey pairs. It was introduced by W. N. Bailey \cite{Bailey1} to give a unified
framework to reprove the Rogers-Ramanujan identities as well as several related
identities of Rogers. The ``Bailey pair machinery'' refers to a combination of
Bailey pairs, Bailey's lemma, and various generalizations that were largely
made famous by G. Andrews. From the Bailey pair machinery,
our proofs require the so-called Bailey chain and Bailey lattice,
rather than just Bailey's lemma.

The full version of the $k$-fold iteration of Bailey's lemma can be found
as Theorem 3.4 in \cite{Andrews1}. We let the parameters 
$N\rightarrow\infty$, $b_1=b_2=\dotsb=b_k\rightarrow\infty$,
$c_1=c_2=\dotsb=c_k=q$, and specialize to Bailey pairs with $a=q$. By doing so
we have that 
\begin{align}\label{EqBaileyChainSpecialized}
&\sum_{n_k\ge n_{k-1}\ge \dotsb \ge n_1\ge0}
\frac{(q)_{n_1} (-1)^{n_1+n_2+\dotsb+n_k} q^{\frac{n_1(n_1+1)}{2}+\frac{n_2(n_2+1)}{2}+\dotsb+\frac{n_k(n_k+1)}{2}} \beta_{n_1}  }
	{(q)_{n_k-n_{k-1}}(q)_{n_{k-1}-n_{k-2}}\dotsm (q)_{n_2-n_1}  }
\nonumber\\
&=
(1-q)\sum_{n\ge0} (-1)^{nk} q^{\frac{kn(n+1)}{2}} \alpha_n
,
\end{align}
for Bailey pairs $(\alpha_n,\beta_n)$ relative to $(q,q)$.

The statement for the Bailey lattice is given in \cite{AgarwalAndrewsBressoud1}
as Theorem 3.1. We let $n\rightarrow\infty$, 
$\rho_1=\rho$, $\rho_j\rightarrow\infty$ for $2\le j\le k$,
$\sigma_j\rightarrow\infty$ for $1\le j\le k$. The identity then
becomes
\begin{align*}
&\sum_{m_1\ge m_2\ge \dotsb \ge m_k\ge 0}
\frac{ (\rho)_{m_1} (-1)^{m_1} \rho^{-m_1} a^{m_1+m_2+\dotsb+m_k} 
	q^{\frac{m_1(m_1-1)}{2}+m_2^2+m_3^2+\dotsb+m_k^2 -m_2-m_3-\dotsb-m_i } \beta_{m_k}   }
{(q)_{m_1-m_2}(q)_{m_2-m_3}\dotsb (q)_{m_{k-1}-m_k}}
\\
&=
	\frac{(a\rho^{-1})_\infty}{(aq)_\infty}
	\sum_{n\ge0} \frac{(\rho)_n (-1)^n \rho^{-n} a^{kn} q^{(k-\frac12)n^2 + (\frac12-i)n  } \alpha_n }
		{(a\rho^{-1})_n(1-aq^{2n})}
	\nonumber\\&\quad
	-
	\frac{(a\rho^{-1})_\infty}{(aq)_\infty}
	\sum_{n\ge1} \frac{(\rho)_n (-1)^n \rho^{-n} a^{kn+i-k+1} q^{(k-\frac12)n^2 + (\frac52+i-2k)n + k-i-2  } \alpha_{n-1} }
		{(a\rho^{-1})_n(1-aq^{2n-2})}
,
\end{align*}
where $0\le i\le k$ and $(\alpha_n,\beta_n)$ is Bailey pair relative
to $(a,q)$.
We further modify this identity by reindexing with 
$m_j=N_j=n_j+n_{j+1}+\dotsb n_k$, which gives
\begin{align}\label{EqBaileyLemmaKIFoldLimit}
&\sum_{n_1,n_2,\dotsc,n_k\ge0}
\frac{ (\rho)_{N_1}(-1)^{N_1}\rho^{-N_1} a^{N_1+N_2+\dotsb+N_k} 
	q^{ \frac{N_1(N_1-1)}{2} +N_2^2+N_3^2+\dots+N_k^2 -N_2-N_3-\dotsb-N_i } \beta_{n_k}  }
{(q)_{n_1}(q)_{n_2}\dotsm (q)_{n_{k-1}}}
\nonumber\\
&=
	\frac{(a\rho^{-1})_\infty}{(aq)_\infty}
	\sum_{n\ge0} \frac{(\rho)_n (-1)^n \rho^{-n} a^{kn} q^{(k-\frac12)n^2 + (\frac12-i)n  } \alpha_n }
		{(a\rho^{-1})_n(1-aq^{2n})}
	\nonumber\\&\quad
	-
	\frac{(a\rho^{-1})_\infty}{(aq)_\infty}
	\sum_{n\ge1} \frac{(\rho)_n (-1)^n \rho^{-n} a^{kn+i-k+1} q^{(k-\frac12)n^2 + (\frac52+i-2k)n + k-i-2  } \alpha_{n-1} }
		{(a\rho^{-1})_n(1-aq^{2n-2})}
.
\end{align}
We note that while it may appear, at first glance, that the right-hand side of \eqref{EqBaileyLemmaKIFoldLimit}
should follow a different pattern when $i=0$, one can verify with
term by term rearrangements that the above is correct.

We require two Bailey pairs. One is the Bailey pair $B(3)$ of Slater \cite{Slater1},
which is relative to $(q,q)$ and defined by
\begin{gather}\label{EqDefBaileyPairB3}
\alpha^{B3}_n := \frac{(-1)^n q^{\frac{n(3n+1)}{2} }(1-q^{2n+1})}{(1-q)}
,\qquad\qquad\qquad
\beta^{B3}_n := \frac{1}{(q)_n}.
\end{gather}
The other Bailey pair we require, which is relative to $(a,q)$, is given by
\begin{gather}\label{Eq:GeneralBaileyPairLDef}
\alpha^\prime_n := \frac{a^n q^{n(n-1)}(1-aq^{2n})}{(1-a)}
,\qquad\qquad
\beta^\prime_n := \frac{1}{(q,a)_n}.
\end{gather}
That $(\alpha_n^\prime,\beta_n^{\prime})$ forms a Bailey pair can be verified
in a variety of ways. One method is to take the defining relation of a Bailey
pair and simplify the sum on $\alpha_j$ with a summation formula for a 
very well poised $_6\phi_5$ \cite[(II.21)]{GasperRahman}.

\section{False Theta Function Identities}

We start with a family of $q$-series identity for ``shifted'' false theta functions. These are similar to Rogers' false 
theta functions, but some of the terms are missing. The first identity is related to the 
$\epsilon=0$ case discussed in the introduction.

\begin{proposition}\label{PropGeneralShifted1/2}
Suppose $i,k,\ell\in\mathbb{Z}$, $k\ge1$, and $0\le i\le k$. Then 
\begin{align*}
&
\frac{ (-q^{\frac12})_\infty}{(q)_\infty}\left( \sum_{n\ge |\ell|} - \sum_{n\le -|\ell|-1} \right)
	q^{ \left(k+\frac12\right)\left(n+\frac{i}{2k+1}\right)^2
		- \frac{\left(2k\ell+i+\ell \right)^2}{4k+2}	}
\\
&=
	\sum_{n_1,n_2,\dotsc,n_k\ge0}
	\frac{ 
		q^{\frac{N_1(N_1+1)}{2}+N_2^2+N_3^2+\dotsb+N_k^2 + (\ell+\frac12)N_1 
			+ (2\ell+1)(N_2+N_3+\dotsb+N_k) - N_{1}-N_{2}-\dotsb-N_{k-i} }
		(-q^{\frac12})_{N_1+\ell}			
	}{(q)_{n_1}(q)_{n_2}\dotsm(q)_{n_k}(q)_{n_k+2\ell}}
.
\end{align*}
\end{proposition}
\begin{proof}
First we verify that the identity for negative $\ell$ follows from the 
identity for positive $\ell$. For this, assume the identity holds for 
positive values of $\ell$. When $\ell$ is negative, $(q)_{n_k+2\ell}^{-1}=0$ for $n_k<-2\ell$, and so
\begin{align*}
&\sum_{n_1,n_2,\dotsc,n_k\ge 0}
	\frac{
		q^{ \frac{N_1(N_1+1)}{2}+N_2^2+N_3^2+\dotsb+N_k^2
			+ (\ell+\frac12) N_1 + (2\ell+1)(N_2+N_3+\dotsb+N_k)
			-N_1 - N_2 - \dotsb - N_{k-i}	}
		(-q^{\frac12})_{N_1+\ell} 		
	}{(q)_{n_1} (q)_{n_2} \dotsm (q)_{n_k} (q)_{n_k+2\ell} }
\\
&=
\sum_{n_1,n_2,\dotsc,n_k\ge 0}
	\frac{
		q^{ \frac{(N_1-2\ell)(N_1+1-2\ell)}{2}+(N_2-2\ell)^2+(N_3-2\ell)^2+\dotsb+(N_k-2\ell)^2}
		(-q^{\frac12})_{N_1-\ell} 		
	}{(q)_{n_1} (q)_{n_2} \dotsm (q)_{n_k} (q)_{n_k-2\ell} }
	\\&\qquad\qquad\qquad\times 			
	q^{ (\ell+\frac12) (N_1-2\ell) + (2\ell+1)(N_2+N_3+\dotsb+N_k-(k-1)2\ell)
			-N_1 - N_2 - \dotsb - N_{k-i} + (k-i)2\ell	}
\\
&=
\sum_{n_1,n_2,\dotsc,n_k\ge 0}
	\frac{
		q^{ \frac{N_1(N_1+1)}{2}+N_2^2+N_3^2+\dotsb+N_k^2
			+ (-\ell+\frac12)N_1 + (-2\ell+1)(N_2+N_3+\dotsb+N_k)
			-N_1 - N_2 - \dotsb - N_{k-i} -2i\ell	}
	}{(q)_{n_1} (q)_{n_2} \dotsm (q)_{n_k} (q)_{n_k-2\ell} }
	\\&\qquad\qquad\qquad\times 			
	(-q^{\frac12})_{N_1-\ell} 
\\
&=
\frac{ (-q^{\frac12})_\infty}{(q)_\infty}\left( \sum_{n\ge -\ell} - \sum_{n\le \ell-1} \right)
	q^{ \left(k+\frac12\right)\left(n+\frac{i}{2k+1}\right)^2
		- \frac{\left(-2k\ell+i-\ell \right)^2}{4k+2} - 2i\ell	}
\\
&=
\frac{ (-q^{\frac12})_\infty}{(q)_\infty}\left( \sum_{n\ge |\ell|} - \sum_{n\le -|\ell|-1} \right)
	q^{ \left(k+\frac12\right)\left(n+\frac{i}{2k+1}\right)^2
		- \frac{\left(2k\ell+i+\ell \right)^2}{4k+2} }
.
\end{align*}
As such, we need only consider the case where $\ell\ge0$.

We prove this identity after the change of variable $i\mapsto k-i$.
We apply \eqref{EqBaileyLemmaKIFoldLimit}
with $\rho\mapsto -q^{\ell+\frac12}$  to the Bailey pair
in \eqref{Eq:GeneralBaileyPairLDef} with $a=q^{2\ell+1}$. Doing so yields
\begin{align*}
&
	\sum_{n_1,n_2,\dotsc,n_k\ge0}
	\frac{
		\left(-q^{\ell+\frac12}\right)_{N_1} 		
		q^{ -(\ell+\frac12)N_1 + (2\ell+1)(N_1+N_2+\dotsb+N_k)
			+ \frac{N_1(N_1-1)}{2}+N_2^2+N_3^2+\dotsb+N_k^2
			- N_2 - N_3 - \dotsb - N_i	}
	}{(q)_{n_1} (q)_{n_2} \dotsm (q)_{n_k} \left(q^{2\ell+1}\right)_{n_k} }
\\
&=
	\frac{(-q^{\ell+\frac12})_\infty}{(q^{2\ell+2})_\infty}
	\sum_{n=0}^\infty
	\frac{
		q^{-(\ell+\frac12)n + (2\ell+1)kn + (k-\frac12)n^2 + (\frac{1}{2}-i)n + n^2 + 2\ell n}
	}{(1-q^{2\ell+1})}
	\\&\quad
	-
	\frac{(-q^{\ell+\frac12})_\infty}{(q^{2\ell+2})_\infty}
	\sum_{n=1}^\infty
	\frac{
		q^{-(\ell+\frac12)n + (2\ell+1)(kn+i-k+1) + (k-\frac12)n^2 + (\frac{5}{2}+i-2k)n + k-i-2+(n-1)^2 + 2\ell (n-1)}
	}{(1-q^{2\ell+1})}
\\
&=
	\frac{(-q^{\ell+\frac12})_\infty}{(q^{2\ell+1})_\infty}
	\sum_{n=0}^\infty
		q^{(k+\frac12)n^2 + (2k\ell+k-i+\ell)n}
	-
	\frac{(-q^{\ell+\frac12})_\infty}{(q^{2\ell+1})_\infty}
	\sum_{n=1}^\infty
		q^{(k+\frac12)n^2 + (2k\ell-k+i+\ell)n + 2\ell(i-k)} 
\\
&=
	\frac{(-q^{\ell+\frac12})_\infty}{(q^{2\ell+1})_\infty}
	\left( \sum_{n\ge\ell} - \sum_{n\le -\ell-1} \right)
	q^{ \left(k+\frac12\right)\left(n +\frac{k-i}{2k+1}\right)^2 
		- \frac{(2k\ell+k-i+\ell)^2}{4k+2}  }
.
\end{align*} 
The claimed identity then follows by multiplying the far extremes of the
above identity by $\frac{(-q^{\frac12})_\ell}{(q)_{2\ell}}$.
\end{proof}

\begin{remark} Proposition \ref{PropGeneralShifted1/2}
gives $q$-hypergeometric expressions for characters for irreducible modules of the $N=1$ super-singlet vertex algebra
(see \eqref{half-char}).
\end{remark}

Next we have a similar identity related to the $\epsilon=\frac12$ case of the series
discussed in the introduction.
\begin{proposition}\label{PropGeneralShifted1}
Suppose $i,k,\ell\in\mathbb{Z}$, $k\ge1$, and $0\le i\le k$. Then 
\begin{align*}
&
\frac{ (-q)_\infty}{(q)_\infty}\left( \sum_{n\ge |\ell|} - \sum_{n\le -|\ell|-1} \right)
	q^{ \left(k+\frac12\right)\left(n+\frac{2i-1}{4k+2}\right)^2
		- \frac{\left(2k\ell+i+\ell-\frac12 \right)^2}{4k+2}	}
	\left(1+q^n\right)
\\
&=
	\sum_{n_1,n_2,\dotsc,n_k\ge0}
	\frac{ 
		q^{\frac{N_1(N_1+1)}{2}+N_2^2+N_3^2+\dotsb+N_k^2 + \ell N_1 
			+ (2\ell+1)(N_2+N_3+\dotsb+N_k) - N_{1}-N_{2}-\dotsb-N_{k-i} }
		(-q)_{N_1+\ell}			
	}{(q)_{n_1}(q)_{n_2}\dotsm(q)_{n_k}(q)_{n_k+2\ell}}
.
\end{align*}
Furthermore, the $i=k$ case simplifies as  
\begin{align*}
&
\frac{ (-q)_\infty}{(q)_\infty}\left( \sum_{n\ge |\ell|} - \sum_{n\le -|\ell|-1} \right)
	q^{ \left(k+\frac12\right)\left(n+\frac{2k-1}{4k+2}\right)^2
		- \frac{\left(2k\ell+k+\ell-\frac12 \right)^2}{4k+2}	}
\\
&=
	\sum_{n_1,n_2,\dotsc,n_k\ge0}
	\frac{ q^{\frac{N_1(N_1+1)}{2}+N_2^2+N_3^2+\dotsb+N_k^2 +\ell N_1 + (2\ell+1)(N_2+N_3+\dotsb+N_k)  } 
		(-q)_{N_1+\ell}}
	{(q)_{n_1}(q)_{n_2}\dotsm(q)_{n_k}(q)_{n_k+2\ell}}
\end{align*}
\end{proposition}
\begin{proof}
As with the proof of Proposition \ref{PropGeneralShifted1/2}, we 
first verify that the identity for negative $\ell$ follows from the 
identity for positive $\ell$. Assuming that the identity holds for positive
values of $\ell$, we then have for negative values of $\ell$ that
\begin{align*}
&\sum_{n_1,n_2,\dotsc,n_k\ge 0}
	\frac{
		q^{ \frac{N_1(N_1+1)}{2}+N_2^2+N_3^2+\dotsb+N_k^2
			+ \ell N_1 + (2\ell+1)(N_2+N_3+\dotsb+N_k)
			-N_1 - N_2 - \dotsb - N_{k-i}	}
		(-q)_{N_1+\ell} 		
	}{(q)_{n_1} (q)_{n_2} \dotsm (q)_{n_k} (q)_{n_k+2\ell} }
\\
&=
\sum_{n_1,n_2,\dotsc,n_k\ge 0}
	\frac{
		q^{ \frac{(N_1-2\ell)(N_1+1-2\ell)}{2}+(N_2-2\ell)^2+(N_3-2\ell)^2+\dotsb+(N_k-2\ell)^2}
		(-q)_{N_1-\ell} 		
	}{(q)_{n_1} (q)_{n_2} \dotsm (q)_{n_k} (q)_{n_k-2\ell} }
	\\&\qquad\qquad\qquad\times 			
	q^{ \ell(N_1-2\ell) + (2\ell+1)(N_2+N_3+\dotsb+N_k-(k-1)2\ell)
			-N_1 - N_2 - \dotsb - N_{k-i} + (k-i)2\ell	}
\\
&=
\sum_{n_1,n_2,\dotsc,n_k\ge 0}
	\frac{
		q^{ \frac{N_1(N_1+1)}{2}+N_2^2+N_3^2+\dotsb+N_k^2
			-\ell N_1 + (-2\ell+1)(N_2+N_3+\dotsb+N_k)
			-N_1 - N_2 - \dotsb - N_{k-i} + \ell -2i\ell	}
		(-q)_{N_1-\ell} 		
	}{(q)_{n_1} (q)_{n_2} \dotsm (q)_{n_k} (q)_{n_k-2\ell} }
\\
&=
\frac{ (-q)_\infty}{(q)_\infty}\left( \sum_{n\ge -\ell} - \sum_{n\le \ell-1} \right)
	q^{ \left(k+\frac12\right)\left(n+\frac{2i-1}{4k+2}\right)^2
		- \frac{\left(-2k\ell+i-\ell-\frac12 \right)^2}{4k+2} + \ell - 2i\ell	}
	(1+q^n)
\\
&=
\frac{ (-q)_\infty}{(q)_\infty}\left( \sum_{n\ge |\ell|} - \sum_{n\le -|\ell|-1} \right)
	q^{ \left(k+\frac12\right)\left(n+\frac{2i-1}{4k+2}\right)^2
		- \frac{\left(2k\ell+i+\ell-\frac12 \right)^2}{4k+2} }
	(1+q^n)
.
\end{align*}
As such, we need only consider the case where $\ell\ge0$.

We prove this identity after the change of variable $i\mapsto k-i$.
We apply \eqref{EqBaileyLemmaKIFoldLimit}
with $\rho\mapsto -q^{\ell+1}$ to the Bailey pair
in \eqref{Eq:GeneralBaileyPairLDef} with $a=q^{2\ell+1}$. Doing so yields
\begin{align*}
&
	\sum_{n_1,n_2,\dotsc,n_k\ge0}
	\frac{
		\left(-q^{\ell+1}\right)_{N_1} 		
		q^{ -(\ell+1)N_1 + (2\ell+1)(N_1+N_2+\dotsb+N_k)
			+ \frac{N_1(N_1-1)}{2}+N_2^2+N_3^2+\dotsb+N_k^2
			- N_2 - N_3 - \dotsb - N_i	}
	}{(q)_{n_1} (q)_{n_2} \dotsm (q)_{n_k} \left(q^{2\ell+1}\right)_{n_k} }
\\
&=
	\sum_{n_1,n_2,\dotsc,n_k\ge 0}
	\frac{
		q^{ \frac{N_1(N_1+1)}{2}+N_2^2+N_3^2+\dotsb+N_k^2
			+ \ell N_1 + (2\ell+1)(N_2+N_3+\dotsb+N_k)
			-N_1 - N_2 - \dotsb - N_i	}
		\left(-q^{\ell+1}\right)_{N_1} 		
	}{(q)_{n_1} (q)_{n_2} \dotsm (q)_{n_k} \left(q^{2\ell+1}\right)_{n_k} }
\\
&=
	\frac{(-q^{\ell})_\infty}{(q^{2\ell+2})_\infty}
	\sum_{n=0}^\infty
	\frac{
		(-q^{\ell+1})_n
		q^{-(\ell+1)n + (2\ell+1)kn + (k-\frac12)n^2 + (\frac{1}{2}-i)n + n^2 + 2\ell n}
	}{(-q^\ell)_n(1-q^{2\ell+1})}
	\\&\quad
	-
	\frac{(-q^{\ell})_\infty}{(q^{2\ell+2})_\infty}
	\sum_{n=1}^\infty
	\frac{
		(-q^{\ell+1})_n
		q^{-(\ell+1)n + (2\ell+1)(kn+i-k+1) + (k-\frac12)n^2 + (\frac{5}{2}+i-2k)n + k-i-2+(n-1)^2 + 2\ell (n-1)}
	}{(-q^\ell)_n(1-q^{2\ell+1})}
\\
&=
	\frac{(-q^{\ell+1})_\infty}{(q^{2\ell+1})_\infty}
	\sum_{n=0}^\infty
		q^{(k+\frac12)n^2 + (2k\ell+k-i+\ell-\frac12)n} (1+q^{n+\ell})
	\\&\quad
	-
	\frac{(-q^{\ell+1})_\infty}{(q^{2\ell+1})_\infty}
	\sum_{n=1}^\infty
		q^{(k+\frac12)n^2 + (2k\ell-k+i+\ell-\frac12)n + 2\ell(i-k)} (1+q^{n+\ell})
\\
&=
	\frac{(-q^{\ell+1})_\infty}{(q^{2\ell+1})_\infty}
	\left( \sum_{n\ge\ell} - \sum_{n\le -\ell-1} \right)
	q^{ \left(k+\frac12\right)\left(n +\frac{2k-2i-1}{4k+2}\right)^2 
		- \frac{(2k\ell+k-i+\ell-\frac12)^2}{4k+2}  }
	(1+q^n)
.
\end{align*} 
The claimed identity then follows. When $i=0$,
\begin{align*}
\left( \sum_{n\ge\ell} - \sum_{n\le -\ell-1} \right)
	q^{ \left(k+\frac12\right)\left(n +\frac{2k-1}{4k+2}\right)^2 
		- \frac{(2k\ell+k+\ell-\frac12)^2}{4k+2} + n }
&=
\left( \sum_{n\ge0} - \sum_{n\ge0} \right)
	q^{ (k+\frac12)n^2  + (2k\ell+k+\ell+\frac12)n + \ell }
=
0,
\end{align*}
which finishes the proof of the proposition.
\end{proof}

\begin{remark} It would be interesting to connect the series in Proposition \ref{PropGeneralShifted1} with characters 
of indecomposable modules for the $N=1$ singlet algebra.
\end{remark}

\section{Modular Identities}

In this part we prove a family of $q$-series identities for certain theta-like series that can be expressed 
as linear combinations of unary theta functions of weight $\frac12$ and $\frac32$.
Such objects appear in the study of {\em logarithmic} modular forms (e.g. modular forms that 
are sums of ordinary modular forms of different weight). This terminology originates 
in Logarithmic Conformal Field Theory, where characters of representations often exhibit this property \cite{AM2, AM1}.

\begin{lemma}\label{shift-help}
Suppose $i,k\in\Z$ and $0\le i\le k$. Then
\begin{align*}
& \sum_{n \in \mathbb{Z}} (2n+1) q^{\frac{2k+1}{2}(n+\frac{i}{2k+1})^2}  
= 
	\sum_{n \in \mathbb{Z}} {\rm sgn}(n)q^{ \frac{2k+1}{2}\left( n+\frac{i}{2k+1}\right)^2}
	+ 
	2 \sum_{\ell \geq 1} \left( \sum_{n\ge\ell} - \sum_{n\le-\ell-1} \right) 
	q^{ \frac{2k+1}{2} \left( n+\frac{i}{2k+1}\right)^2}
.
\end{align*}
\end{lemma}
\begin{proof}
We omit the proof as it is near trivial.
\end{proof}

Let $k \geq 2$ and let $A$ be the $(k+1) \times (k+1)$ matrix defined by
$$A:=\begin{pmatrix} 
\frac{3}{2} & -\frac12 & 0 & ... & ... & ... & ... 
\\[0.5ex] 
-\frac12 & 1 & -\frac12 & ... & ... & ... & ...
\\[0.5ex] 
0 & -\frac12 & 1 & ... & ... & ... & ...
\\ 
... & ... & ... & ... & ... & ... & ... 
\\
... & ... & ... & ...& 1 & -\frac12 & -\frac12 
\\[0.5ex] 
... & ... & ... & ...& -\frac12 & 1 & 0 
\\[0.5ex] 
... & ... & ... & ...& -\frac12 & 0 & 1 
\end{pmatrix}.$$
We note that 
$$A=\frac{1}{2}\left(D_{k+1} + E_{1,1}\right),$$ 
where $D_{k+1}$ is the Cartan matrix of $D$-type. It is easy to show
\begin{equation}\label{A-inverse}
A^{-1}=
\begin{pmatrix} 
	1 & 1 & 1 & ... & 1 & \frac12 & \frac12 
	\\[0.5ex]  
	1 & 3 & 3 & ... & 3 & \frac32 & \frac32 
	\\[0.5ex]
	1 & 3 & 5 & ... & 5 & \frac52 & \frac52 
	\\[0.5ex] 
	... & ... & ... & ... & ... & ... & ... 
	\\[0.5ex] 
	1 & 3 & 5 & ...& 2k-3 & \frac{2k-3}{2} & \frac{2k-3}{2} 
	\\[0.5ex] 
	\frac12 & \frac32 & \frac52 & ...& \frac{2k-3}{2} & \frac{2k+1}{4} & \frac{2k-3}{4} 
	\\[0.5ex] 
	\frac{1}{2} & \frac32 & \frac52 & ...& \frac{2k-3}{2} & \frac{2k-3}{4} & \frac{2k+1}{4} 
\end{pmatrix}
\end{equation}
or in more compact form
$$(A^{-1})_{i,j}=\left\{ \begin{array}{cc} {\rm min}\{2i-1,2j-1\} , &  1 \leq i,j \leq k-1  \\   \frac{2j-1}{2},  &      1 \leq j \leq k-1, k \leq i \leq k+1   \\   \frac{2k+1}{4}, &  i=j \in \{k, k+1 \} \\ \frac{2k-3}{4}, & i=k, j=k+1 \end{array} \right.$$
For $k=1$, we let $A^{-1}=\begin{pmatrix} \frac34 & -\frac14 \\ -\frac14 & \frac34 \end{pmatrix}$.

Now we can state G\"ollnitz-Gordon-Andrews-type identities for the series in Lemma \ref{shift-help}.
\begin{theorem} \label{NS-char-log} 
Suppose $i\in\Z$, $k \in \mathbb{N}$, and $0\le i\le k$. Then
\begin{multline*}
q^{-\frac{i^2}{2(2k+1)}} \frac{(-q^{1/2})_\infty}{(q)_\infty} 
	\sum_{n \in \mathbb{Z}} (2n+1) q^{\frac{2k+1}{2}(n+\frac{i}{2k+1})^2}
\\	
=
	\sum_{n_1,n_2,\dotsc,n_{k+1} \geq 0 \atop n_{k} \equiv n_{k+1} \mod 2} 
	\frac{q^{\frac{1}{2}{\bf n} \cdot A^{-1} \cdot {\bf n}^T + n_{k-i+1} + 2n_{k-i+2}+\dotsb + (i-1)n_{k-1}+\frac{i}{2}(n_{k}+n_{k+1})} 
		(-q^{\frac12})_{n_1+n_2+\dotsb + n_{k-1}+ \frac{n_{k}+n_{k+1}}{2}}}
	{(q)_{n_1} (q)_{n_2} \cdots (q)_{n_{k+1}}}
,
\end{multline*}
where ${\bf n}=(n_1,...,n_{k+1})$.
\end{theorem}
\begin{proof}
We first rewrite the identity in Proposition \ref{PropGeneralShifted1/2} as
\begin{align*}
&F_\ell(q) := 
	q^{-\frac{i^2}{4k+2}}
	\frac{ (-q^{\frac12})_\infty}{(q)_\infty}
	\left( \sum_{n\ge |\ell|} - \sum_{n\le -|\ell|-1} \right)
		q^{ \left(k+\frac12\right)\left(n+\frac{i}{2k+1}\right)^2 }
\\
&=
	\sum_{n_1,n_2,\dotsc,n_k\ge0}
	\frac{
		q^{(k+\frac12)\ell^2 + i\ell + \frac{N_1(N_1+1)}{2}+N_2^2+N_3^2+\dotsb+N_k^2 + (\ell+\frac12)N_1 
			+ (2\ell+1)(N_2+N_3+\dotsb+N_k) - N_{1}-N_{2}-\dotsb-N_{k-i} }
	}{(q)_{n_1}(q)_{n_2}\dotsm(q)_{n_k}(q)_{n_k+2\ell}}
	\\&\qquad\qquad\times
	(-q^{\frac12})_{N_1+\ell} 
.
\end{align*}

From formula (\ref{A-inverse}) for $A^{-1}$, we easily get 
\begin{align*}
\frac{1}{2}(n_1,\dotsc,n_{k-1},0,0) \cdot A^{-1} \cdot (n_1,\dotsc,n_{k-1},0,0)^T
&=
	\frac{1}{2} \left(\sum_{i=1}^{k-1} n_i\right)^2 
	+ 
	\sum_{j=2}^{k-1}\left(\sum_{i=j}^{k-1} n_i\right)^2
,\\ 
(0,\dotsc,0,n_{k},n_{k}+2 \ell) \cdot A^{-1} \cdot (n_1,\dotsc,n_{k-1},0,0)^T
&= 
	\frac{1}{2}\sum_{i=1}^{k-1}
	(2i-1)n_i(2n_k+2\ell)
,\\ 
\frac{1}{2}(0,\dotsc,0,n_{k},n_{k}+2 \ell) \cdot A^{-1} \cdot (0,\dotsc,0,n_k,n_{k}+2 \ell)^T
&=
	\frac{2k+1}{8}\left( n_k^2 +(n_k+2\ell)^2\right)
	\\&\quad
	+\frac{2k-3}{4}n_k(n_k+2\ell)
. 
\end{align*}
We then have that
\begin{align*}
&\tfrac12(n_1,n_2,\dotsc,n_k,n_k+2\ell)  A^{-1} (n_1,n_2,\dotsc,n_k,n_k+2\ell)^T
\\&=
	\tfrac12 N_1^2 + N_2^2 + N_3^2 +\dotsb +N_k^2 
	+\ell N_1 + 2\ell(N_2+N_3+\dotsb+N_k)
	+ (k+\tfrac12)\ell^2
.
\end{align*}
As such,
\begin{align*}
&\sum_{\substack{n_1,n_2,\dotsc,n_{k+1}\ge0 \\ n_k\equiv n_{k+1}\pmod{2} }}
\frac{
	q^{\frac{1}{2}{\bf n} \cdot A^{-1} \cdot {\bf n}^T 
		+ n_{k-i+1} + 2n_{k-i+2} + \dotsb + (i-1)n_{k-1} + \frac{i}{2}(n_k+n_{k+1})  }
	(-q^{\frac12})_{n_1+n_2+\dotsb+n_{k-1}+\frac{n_k+n_{k+1}}{2} }
}{(q)_{n_1}(q)_{n_2}\dotsm (q)_{n_{k+1}}}
\\
&=
	\sum_{\substack{n_1,n_2,\dotsc,n_k\ge0\\ \ell\in\mathbb{Z}} }
	\frac{
		q^{ \frac12 N_1^2 + N_2^2 + N_3^2 +\dotsb N_k^2 + \ell N_1 + 2\ell(N_2+N_3+\dotsb+N_k)
			+ (k+\frac12)\ell^2
			+ n_{k-i+1} + 2n_{k-i+2} + \dotsb + (i-1)n_{k-1} + in_k + i\ell  }
	}{(q)_{n_1}(q)_{n_2}\dotsm (q)_{n_k} (q)_{n_k+2\ell}}
	\\&\qquad\qquad\times
	(-q^{\frac12})_{N_1+\ell }
\\
&=
	\sum_{\substack{n_1,n_2,\dotsc,n_k\ge0\\ \ell\in\mathbb{Z}}}
	\frac{
		q^{(k+\frac12)\ell^2 + i\ell + \frac{N_1(N_1+1)}{2}+N_2^2+N_3^2+\dotsb+N_k^2 + (\ell+\frac12)N_1 
			+ (2\ell+1)(N_2+N_3+\dotsb+N_k) - N_{1}-N_{2}-\dotsb-N_{k-i} }
	}{(q)_{n_1}(q)_{n_2}\dotsm(q)_{n_k}(q)_{n_k+2\ell}}
	\\&\qquad\qquad\times
	(-q^{\frac12})_{N_1+\ell} 
\\
&=
	\sum_{ \ell\in\mathbb{Z}}F_\ell(q)
=
	q^{-\frac{i^2}{4k+2}}\frac{(-q^{\frac12})_\infty}{(q)_\infty}
	\sum_{n\in\mathbb{Z}} (2n+1) q^{(k+\frac12)\left(n+\frac{i}{2k+1}\right)^2} 
,
\end{align*}
where the final equality is by Lemma \ref{shift-help}.
This proves the claim.
\end{proof}

We record the simplest case with $i=k=1$.
\begin{corollary} \label{log-1}
The following identity holds,
$$q^{-1/6} \frac{(-q^{1/2})_\infty}{(q)_\infty} \sum_{n \in \mathbb{Z}} (2n+1) q^{\frac32(n+\frac13)^2}
=\sum_{n_1,n_2 \geq 0 \atop n_1 \equiv n_2 \mod 2}  
\frac{q^{\frac38 n_1^2 + \frac38 n_2^2 - \frac14 n_1n_2 +\frac12(n_1+n_2)} (-q^{1/2})_{\frac{n_1+n_2}{2}}}{(q)_{n_1} (q)_{n_2}}.$$
\end{corollary}

\begin{remark}
Notice that for $i=0$, the sum on the left-hand side side in Theorem \ref{NS-char-log} simplifies, so we can write
$$ \frac{(-q^{1/2})_\infty}{(q)_\infty} 
	\sum_{n \in \mathbb{Z}}  q^{\left(k+\frac12 \right) n^2}
=\sum_{n_1,n_2,\dotsc,n_{k+1} \geq 0 \atop n_{k} \equiv n_{k+1} \mod 2} 
	\frac{q^{\frac{1}{2}{\boldsymbol{n}}^T \cdot A^{-1} \cdot {\boldsymbol{n}}} (-q^{\frac12})_{n_1+n_2+\cdots + n_{k-1}+ \frac{n_{k}+n_{k+1}}{2}}}
	{(q)_{n_1}(q)_{n_2} \cdots (q)_{n_{k+1}}}.$$
It would be interesting to find a combinatorial description of this identity. 
\end{remark}

There are similar identities coming from Proposition \ref{PropGeneralShifted1}.
The $i=k$ case is more elegant than the general case, and so we record it as
a separate theorem below. In both cases, we omit the proofs as they are near
identical to that of Theorem \ref{NS-char-log}.

\begin{theorem} \label{log-Ramond}
For $k\in\N$, 
\begin{align*}
&q^{-\frac{(2k-1)^2}{8(2k+1)}} \frac{(-q)_\infty}{(q)_\infty} \sum_{n \in \mathbb{Z}} (2n+1) q^{\frac{2k+1}{2}(n+\frac{2k-1}{2(2k+1)})^2}
\\
&=\sum_{n_1,n_2,\dotsc,n_{k+1} \geq 0 \atop n_{k} \equiv n_{k+1} \mod 2} 
\frac{q^{\frac{1}{2}{\boldsymbol{n}}^T \cdot A^{-1} \cdot {\boldsymbol{n}}
	+\frac12 n_1+\frac32 n_2+\cdots +\frac{2k-3}{2} n_{k-1}+\frac{2k-1}{4}(n_{k}+n_{k+1})} 
(-q)_{n_1+n_2+\cdots + n_{k-1}+ \frac{n_{k}+n_{k+1}}{2}}}{(q)_{n_1} (q)_{n_2} \cdots (q)_{n_{k+1}}}.
\end{align*}
\end{theorem}

For general $0\le i\le k$, we instead have the following.
\begin{theorem}\label{log-Ramond2}
Suppose $i\in\Z$, $k\in\N$, and $0\le i\le k$, then
\begin{align*}
&q^{-\frac{(2i-1)^2}{8(2k+1)}} \frac{(-q)_\infty}{(q)_\infty} 
	\sum_{n \in \mathbb{Z}} (2n+1) q^{\frac{2k+1}{2}(n+\frac{2i-1}{4k+2})^2}
	(1+q^n)
\\	
&=
	\sum_{n_1,n_2,\dotsc,n_{k+1} \geq 0 \atop n_{k} \equiv n_{k+1} \mod 2} 
	    q^{ \frac{1}{2}\boldsymbol{n} \cdot A^{-1} \cdot \boldsymbol{n}^T 
	        -\frac{1}{2}(n_1+n_2+\dotsb+n_{k-i})
	        +\frac{1}{2}n_{k-i+1}+\frac{3}{2}n_{k-i+2}+\dotsb+\frac{2i-3}{2}n_{k-1} 
            + \frac{2i-1}{4}(n_k+n_{k+1}) }	
	\\[-5ex]&\qquad\qquad\qquad\quad
	\times
	\frac{ (-q)_{n_1+n_2+\dotsb + n_{k-1}+ \frac{n_{k}+n_{k+1}}{2}}}
	{(q)_{n_1} (q)_{n_2} \cdots (q)_{n_{k+1}}}
.
\end{align*}
\end{theorem}

\section{further $q$-series identities}

As discussed in the introduction, identity (\ref{euler-p2}),
which can be obtained from the Fourier expansion of a meromorphic Jacobi form, gives an elegant 
expression for the Rogers false theta function.
In this section, we first present a generalization of that identity.
\begin{theorem} \label{euler-false-p} For $k\in\N$,
$$\frac{\sum_{n \in \mathbb{Z}} {\rm sgn}(n) q^{(k+1)n^2+kn}}{(q)^{2k}_\infty}
=\sum_{n_1,n_2,\dotsc,n_{2k-1} \geq 0} \frac{q^{\sum_{i=1}^{2k-2} n_i n_{i+1}+\sum_{i=1}^{2k-1} n_i}}{(q)_{n_1}^2 (q)_{n_2}^2 \cdots (q)_{n_{2k-1}}^2}.$$
\end{theorem}
We note that these series have an odd number of summation variables. 
Interestingly, with an even number of summation variables we get a family of  {\em modular} identities. 
\begin{theorem} \label{euler-modular-p}
For $k\in\N$, 
$$\frac{(q,q^{2k+2},q^{2k+3};q^{2k+3})}{(q)^{2k+1}_\infty}
=\sum_{n_1,n_2,\dotsc,n_{2k} \geq 0} \frac{q^{\sum_{i=1}^{2k-1} n_i n_{i+1}+\sum_{i=1}^{2k} n_i}}{(q)_{n_1}^2 (q)_{n_2}^2 \cdots (q)_{n_{2k}}^2}.$$
\end{theorem} 
We should point out that Theorem \ref{euler-modular-p} was first conjectured in \cite{CS} based on
an analysis of Schur's indices in certain Argyres-Douglas theories in physics.
We shall present a uniform proof of both theorems. We do so with
a combination of $q$-series techniques and Bailey's lemma.

\begin{lemma}\label{LemmaMultiSum1}For $k\in\N$,
\begin{align*}
&\sum_{n_1,n_2,\dotsc,n_k\ge0}
\frac{q^{n_1n_2 + n_2n_3 + \dotsb + n_{k-1}n_k + n_1+n_2+\dotsb+n_k}}
	{(q)_{n_1}^2(q)_{n_2}^2\dotsm (q)_{n_k}^2}
\\
&=
	\frac{1}{(q)_\infty^{k+1}}
	\sum_{m_1,m_2,\dotsc,m_k\ge0}
	\frac{(-1)^{m_1+m_2+\dotsb+m_k} 
		q^{\frac{m_1(m_1+1)}{2}+\frac{m_2(m_2+1)}{2}+\dotsb+\frac{m_k(m_k+1)}{2}} }
	{(q)_{m_1-m_2}(q)_{m_2-m_3}\dotsm (q)_{m_{k-1}-m_k}}
.
\end{align*}
\end{lemma}
\begin{proof}
We prove the stronger result that 
\begin{align}\label{EqLemmaMultiSum1BigId}
&\sum_{n_1,n_2,\dotsc,n_k\ge0}
\frac{ x_2^{n_1}x_3^{n_2}\dotsm x_k^{n_{k-1}} q^{n_1n_2+n_2n_3+\dotsb + n_{k-1}n_k + n_1+n_2+\dotsb+n_k} }
	{ (q,x_1q)_{n_1} (q,x_2q)_{n_2} \dotsm (q,x_kq)_{n_k} }
\nonumber\\
&=
	\frac{1}{(q)_\infty}
	\sum_{m_1,m_2,\dotsc, m_k\ge0}
	\frac{(-1)^{m_1+m_2+\dotsb+m_k}x_1^{m_1}x_2^{m_2}\dotsm x_k^{m_k} 
		q^{\frac{m_1(m_1+1)}{2}+\frac{m_2(m_2+1)}{2}+\dotsb+\frac{m_k(m_k+1)}{2}}	
	}{(q)_{m_1-m_2}(q)_{m_2-m_3}\dotsm(q)_{m_{k-1}-m_k} (x_1q,x_2q,\dotsc,x_kq)_\infty}	
	.
\end{align}
The lemma then follows from by setting $x_1=x_2=\dotsb=x_k=1$. Multiplying the
left-hand side of \eqref{EqLemmaMultiSum1BigId} by $(x_1q,x_2q,\dotsc,x_kq)_\infty$
gives
\begin{align}\label{EqLemmaMultiSum1LongId}
&\sum_{n_1,n_2,\dotsc,n_k\ge0}
	\frac{ x_2^{n_1}x_3^{n_2}\dotsm x_k^{n_{k-1}} 
		q^{n_1n_2+n_2n_3+\dotsb + n_{k-1}n_k + n_1+n_2+\dotsb+n_k} 
		(x_1q^{n_1+1},x_2q^{n_2+1},\dotsc,x_kq^{n_k+1})_\infty
	}{ (q)_{n_1} (q)_{n_2} \dotsm (q)_{n_k} }
\nonumber\\
&=
\sum_{\substack{ n_1,n_2,\dotsc,n_k\ge0\\m_1,m_2,\dotsc,m_k\ge0  }}
	\frac{ (-1)^{m_1+m_2+\dotsb+m_k}
		x_1^{m_1}x_2^{m_2+n_1}x_3^{m_3+n_2}\dotsm x_k^{m_k+n_{k-1}} 
		q^{n_1n_2+n_2n_3+\dotsb + n_{k-1}n_k }
	}{ (q)_{n_1} (q)_{n_2} \dotsm (q)_{n_k}(q)_{m_1} (q)_{m_2} \dotsm (q)_{m_k} }
	\nonumber\\&\qquad\qquad\qquad\times\!
		q^{\frac{m_1(m_1+1)}{2}+\frac{m_2(m_2+1)}{2}+\dotsb+\frac{m_k(m_k+1)}{2}
			+ n_1(m_1+1) + n_2(m_2+1) + \dotsb + n_k(m_k+1) 
		}
\nonumber\\
&=
\sum_{\substack{ n_1,n_2,\dotsc,n_k\ge0\\m_1,m_2,\dotsc,m_k\ge0  }}
	\frac{ (-1)^{m_1+m_2+\dotsb+m_k+n_1+n_2+\dotsb+n_{k-1}}
		x_1^{m_1}x_2^{m_2}\dotsm x_k^{m_k} 
		q^{\frac{m_1(m_1+1)}{2}+\frac{m_2(m_2+1)}{2}+\dotsb+\frac{m_k(m_k+1)}{2}}
	}{ (q)_{n_1} (q)_{n_2} \dotsm (q)_{n_k}(q)_{m_1} (q)_{m_2-n_1}(q)_{m_3-n_2} \dotsm (q)_{m_k-n_{k-1}} }
	\nonumber\\&\qquad\qquad\qquad\times\!
		q^{ \frac{n_1(n_1+1)}{2}+\frac{n_2(n_2+1)}{2}+\dotsb+\frac{n_{k-1}(n_{k-1}+1)}{2}
			+ n_{k} + m_1n_1 + m_2(n_2-n_1) + m_3(n_3-n_2) + \dotsb + m_{k}(n_k-m_{k-1}) 
		}
,
\end{align}
where in the first equality we have expanded each $(x_jq)_\infty$ according to 
an identity of Euler,
\begin{gather*}
(z;q)_\infty
=
\sum_{m\ge0} \frac{(-1)^m z^m q^{\frac{m(m-1)}{2}}}{(q)_m}
,
\end{gather*}
and in the second equality we have shifted the indices $m_j\mapsto m_j-n_{j-1}$
for $j\ge2$ (since $(q)_n^{-1}=0$ when $n$ is negative, the index bounds are 
still valid). 

We recall a certain finite analogue of the $q$-binomial theorem \cite[Theorem 3.3]{Andrews2} states that
\begin{gather*}
\sum_{n=0}^m \frac{(-1)^n z^n q^{\frac{n(n-1)}{2}} (q)_m }{(q)_n (q)_{m-n}} 
= (z)_m.
\end{gather*}
As such the inner sums on $n_j$ in the final equality of 
\eqref{EqLemmaMultiSum1LongId}, for $1\le j<k$, are
\begin{gather*}
\sum_{n_j\ge0  }
	\frac{ (-1)^{n_j} q^{ n_j(m_j-m_{j+1}+1) + \frac{n_j(n_j-1)}{2}} }
	{ (q)_{n_j}(q)_{m_{j+1}-n_j}}
=
	\frac{ (q^{m_j-m_{j+1}+1})_{m_{j+1}} }{(q)_{m_{j+1}}}
=
	\frac{(q)_{m_j}}{(q)_{m_{j+1}}(q)_{m_j-m_{j+1}}}
.
\end{gather*}
The inner sum on $n_k$, by another identity of Euler, is instead
\begin{gather*}
\sum_{n_k\ge0} \frac{q^{n_k(m_k+1)}}{(q)_{n_k}}
=
	\frac{1}{(q^{m_k+1})_\infty}
.
\end{gather*}
Equation \eqref{EqLemmaMultiSum1BigId} then follows by plugging these
back into \eqref{EqLemmaMultiSum1LongId} and performing some elementary rearrangements
of the summands. 
\end{proof}

We are now in a position to state and prove the general identity that encapsulates
both Theorem \ref{euler-false-p} and Theorem \ref{euler-modular-p}.

\begin{theorem}\label{TheoremMultiSum} For $k\in\N$,
\begin{align*}
&\sum_{n_1,n_2,\dotsc,n_k\ge0}
\frac{q^{n_1n_2 + n_2n_3 + \dotsb + n_{k-1}n_k + n_1+n_2+\dotsb+n_k}}
	{(q)_{n_1}^2(q)_{n_2}^2\dotsm (q)_{n_k}^2}
\\
&=
	\frac{1}{(q)_\infty^{k+1}}
	\left(\sum_{n\ge0}+(-1)^{k}\sum_{n<0}\right)
	(-1)^{n(k+1)} q^{ \frac{(k+3)n^2+(k+1)n}{2}}	
.
\end{align*}
\end{theorem}
\begin{proof}
By applying the version of Bailey's lemma stated in \eqref{EqBaileyChainSpecialized}
to the Bailey pair defined in \eqref{EqDefBaileyPairB3},
we find that Lemma \ref{LemmaMultiSum1} implies
\begin{align*}
&\sum_{n_1,n_2,\dotsc,n_k\ge0}
\frac{q^{n_1n_2 + n_2n_3 + \dotsb + n_{k-1}n_k + n_1+n_2+\dotsb+n_k}}
	{(q)_{n_1}^2(q)_{n_2}^2\dotsm (q)_{n_k}^2}
=
	\frac{1}{(q)_\infty^{k+1}}
	\sum_{n\ge0}
	(-1)^{n(k+1)} q^{\frac{(k+3)n^2+(k+1)n}{2}}(1-q^{2n+1})
\\
&=
	\frac{1}{(q)_\infty^{k+1}}
	\left(\sum_{n\ge0}+(-1)^k\sum_{n<0}\right)
	(-1)^{n(k+1)} q^{\frac{(k+3)n^2+(k+1)n}{2}}
.
\end{align*}
\end{proof}

We find that Theorem \ref{TheoremMultiSum} establishes Thereom \ref{euler-false-p} and \ref{euler-modular-p}.
In particular, Theorem \ref{euler-false-p} immediately follows from the statement of the theorem by letting 
$k\mapsto 2k-1$. For Theorem
\ref{euler-modular-p}, we let $k\mapsto 2k$ in Theorem \ref{TheoremMultiSum} and sum the series
by using the Jacobi triple product identity to find that
\begin{align*}
\sum_{n\in\Z} (-1)^{n} q^{ \frac{(2k+3)n^2+(2k+1)n}{2}}	
&=
\sum_{n\in\Z} (-1)^{n} q^{ \frac{(2k+3)n^2-(2k+1)n}{2}}	
=
	(q,q^{2k+2},q^{2k+3};q^{2k+3})_\infty
.
\end{align*}


\section{Classical Dilogarithm identities}

It is known that modular $q$-hypergeometric identities  give rise to identities for the Rogers dilogarithm function  
\cite{Kirillov,Nahm1,Nahm2,Zagier}; see also \cite{VZ}. This is the case with characters of modules of rational 
vertex algebras (or conformal field theories) as they can be often expressed as $n$-fold $q$-hypergeometric series. 

Let $A=(A_{i,j})$ be a symmetric positive definite $k \times k$ matrix with rational entries, $B \in \mathbb{Q}^k$, $C \in \mathbb{Q}$, and 
$$F_{A,B,C} (q)
:=
\sum_{{\boldsymbol{n}}=(n_1,n_2,\dotsc,n_k) \geq 0} 
\frac{q^{\frac{1}{2} \boldsymbol{n}^T A \boldsymbol{n}+B{\boldsymbol{n}}+C}}{(q)_{n_1}(q)_{n_2} \cdots (q)_{n_k}}.$$
This is an example of a multi $q$-hypergeometric series sometimes called a {\em Nahm sum}.
According to \cite{KKMM,Kirillov, Nahm1,Nahm2}, using the saddle point method, one can show that 
its asymptotic expansion (as $t \to 0^+$, $q=e^{-t}$) is given by:
\begin{gather*}
F_{A,B,C}(e^{-t}) e^{-\alpha/t} \sim \beta e^{-\gamma t}(1 + O(t))
,\\
\alpha:=\sum_{i=1}^k (L(1)-L(Q_i)),
\end{gather*}
where 
$${\rm L}(x):={\rm Li}_2(x)+\frac{1}{2}\log(x)\log(1-x) , \ x \in (0,1),$$
denotes Rogers' dilogarithm function and  $${\rm Li}_2(x)=\sum_{n \geq 1} \frac{x^n}{n^2}$$ is the classical dilogarithm \footnote{$L(x)$ can be  extended to $x=0$ and $x=1$; $L(0)=0$, $L(1)=\frac{\pi^2}{6}$}, and 
$Q_i  \in (0,1)$ are unique solutions of the TBA
(thermodynamic Bethe ansatz) equation 
\begin{equation} \label{TBA-org}
(1-Q_i)=\prod_{j=1}^k Q_j^{A_{i,j}}.
\end{equation}
A detailed treatment of this result can be found in 
M. Vlasenko and S. Zwegers's recent paper \cite{VZ}.

We would like to extend this idea to our false theta function identities. 

\begin{lemma} Let $A$ be a positive definite real symmetric $k \times k$ matrix. Then
\begin{equation} \label{TBA-org-log}
(1-Q_k)^2=\prod_{j=1}^k Q_j^{A_{k,j}},    \ \ (1-Q_i)=\prod_{j=1}^k Q_j^{A_{i,j}},  \ 1 \leq i \leq k-1
\end{equation}
has a unique solution inside $(0,1)$.
\end{lemma}
\begin{proof}
The proof follows along the lines of \cite[Lemma 2.1]{VZ}.
The only modification required is to choose $$f_A({\bf x})=\frac12 {\bf x}^T A {\bf x} + 2 {\rm  Li}_2({\rm Exp}(-x_k))+\sum_{i=1}^{k-1} {\rm  Li}_2({\rm Exp}(-x_i)),$$ 
and proceed as in \cite{VZ}.
\end{proof}

\subsection{Even case}
Recall (for $k \geq 1$ )
\begin{equation} \label{BM}
F(q):=\frac{q^{-\frac{k^2}{4(k+1)}}}{(q;q)_\infty} {\sum_{n \in \mathbb{Z}} {\rm sgn}(n) q^{(k+1)(n+\frac{k}{2(k+1)})^2}}
=\sum_{n_1,n_2,\dotsc,n_{k} \geq 0} \frac{q^{N_1^2+N_2^2+\cdots+N_{k}^2+N_1+N_2+\cdots + N_{k}}}
{(q)_{n_{k}}^2 (q)_{n_1} (q)_{n_2} \cdots (q)_{n_{k-1}}}.
\end{equation}
We let $f(q)=\sum_{n \in \mathbb{Z}} {\rm sgn}(n) q^{k(n+\frac{k-1}{2k})^2}$.
It is known that (as $ t \to 0^+$) (e.g. \cite[Proposition 6.5]{Zagier1})
$$f(e^{-t}) \sim \frac{1}{k} + O(t).$$
Therefore 
$$e^{-\frac{\pi^2}{6t}} F(e^{-t}) \sim \sqrt{\frac{t}{2\pi}}  \left( \frac{1}{k} +O(t) \right),$$
where we have used the well-known behavior $\eta(e^{-t})=q^{1/24}(q)_\infty |_{q=e^{-t}}  \sim  \sqrt{\frac{2\pi}{t}} e^{-\frac{\pi^2}{6t}}$.
Now we look at the $q$-hypergeometric side. Because of the $\frac{1}{(q)_{n_{k}}^2}$  we are led to a system of $k \times k$ equations (cf. \cite{VZ}) 
\begin{align}
\label{TBA}
 (1-Q_i) &=\prod_{j \geq 1} Q_{j}^{B_{i,j}}, \ \ 1 \leq i \leq k-1, \\
  (1-Q_k)^2 & =\prod_{j \geq 1} Q_{j}^{B_{k,j}}, \nonumber
\end{align}
where $B=(B_{i,j})=(2{\rm min}(i,j))$ is the matrix
corresponding to the quadratic form in the exponent of the RHS in (\ref{BM}).

Applying the procedure in \cite{VZ} for computing the asymptotic exponent 
$\alpha$ of $F(q)$ in terms of $L(x)$, with slight modifications due to the squared factor $(q)_{n_{k}}^2$,  gives
the following identity.
\begin{proposition} \label{dilog-log} For $k \geq 2$, let $\{ Q_i \}$ denote the unique solution of (\ref{TBA}) 
inside the interval $(0,1)$. Then 
\begin{equation} \label{TBA-Dilog}
2(L(1)- L(Q_k))+\sum_{r \geq 1}^{k-1} (L(1)-L(Q_r))=\frac{\pi^2}{6}.
\end{equation}
\end{proposition}
\begin{proof}  It is not hard to see that $Q_k=\frac{k}{k+1} $, $Q_1=1-\frac{1}{2^2},...,Q_{k-1}=1-\frac{1}{k^2}$ is the solution of (\ref{TBA}) inside $(0,1)$.
Using $L(x)+L(1-x)=L(1)$, identity (\ref{TBA-Dilog}) is equivalent to 
\begin{equation} \label{kirillov}
2L \left(\frac{1}{k+1}\right)+\sum_{r \geq 2}^k L\left(\frac{1}{r^2}\right)=\frac{\pi^2}{6}.
\end{equation}
The last identity is known and mentioned in \cite[Exercise 6]{Kirillov}. It can be easily proven 
by induction using the relation
\begin{equation} \label{Rog-help}
{L}\left(1-\frac{1}{x^2}\right)+2
   \left({L}\left(\frac{1}{x}\right)+{L}\left(\frac{x}{x+1}\right)\right)=\frac{\pi^2}{2},
   \end{equation}
with $x=k$, a consequence of the Five-Term relation for Rogers' dilogarithm \cite[Formula (1.4)]{Kirillov}.
\end{proof}


\subsection{$N=1$ super identities}
In this part we connect the $q$-series identities
\begin{eqnarray}
\label{CM} 
&& q^{-\frac{k^2}{2(2k+1)}} \frac{(-q^{1/2})_\infty}{(q)_\infty} \sum_{n \in \mathbb{Z}} {\rm sgn}(n) q^{\frac{2k+1}{2} \left(n+\frac{k}{2k+1}\right)^2} \nonumber  \\
&&=\sum_{n_1,n_2,\dotsc,,n_k \geq 0}  
\frac{q^{\frac{N_1^2}{2}+N_2^2 + N_3^2 + \cdots + N_{k}^2+N_1+N_2+\cdots + N_{k}} (-q^{1/2})_{N_1}}
	{(q)_{n_k}^2 (q)_{n_1} (q)_{n_2} \cdots (q)_{n_{k-1}}} 
\end{eqnarray}
with dilogarithm identities. Here the situation is slightly more complicated due to 
the additional $q$-Euler factor, $(-q^{1/2})_{N_1}$,  appearing in the numerator. Observe that the leading asymptotics of the left-hand side in (\ref{CM}) is easily determined. After letting $q=e^{-t}$, it grows as $C e^{\frac{\pi^2}{4 t}}$ as $t \to 0^+$, which follows 
from using $q^{-\frac{1}{48}} (-q^{\frac12})_\infty =\frac{\eta(\tau)^2}{\eta(\frac{\tau}{2}) \eta(2 \tau)}$.
 Therefore we expect a family of identities for Rogers' dilogarithm as in (\ref{kirillov}), where the right-hand side equals $\frac{\pi^2}{4}$. 
\begin{proposition}\label{PropositionMoreDilogarithms1}
For $k \in \mathbb{N}$,
\begin{equation} \label{super-kirillov}
2L \left(\frac{2}{2k+1} \right)+\sum_{r = 1}^{k-1} L\left(\frac{4}{(2r+1)^2}\right)+L\left(\frac{3}{4}\right)-L\left(\frac12\right)=\frac{\pi^2}{4}.
\end{equation}
\end{proposition}
\begin{proof} As in \cite{KKMM}, we first write
\begin{align*}
F(q)&:= \sum_{n_1,n_2,\dotsc,n_k \geq 0}  
	\frac{q^{\frac{N_1^2}{2}+N_2^2 + N_3^2 + \cdots + N_{k}^2+N_1+N_2+\cdots + N_{k}} (-q^{1/2})_{N_1}}
	{(q)_{n_k}^2 (q)_{n_1}(q)_{n_2} \cdots (q)_{n_{k-1}}} \\
&= 
\sum_{n_1,n_2,\dotsc,n_{k+1} \geq 0}  
\frac{q^{N_1^2+N_2^2 + \cdots + N_{k}^2+N_1+N_2+\cdots + N_{k}-N_{1}n_{k+1} +\frac{n_{k+1}^2}{2}}}
	{(q)_{n_k}^2 (q)_{n_1}(q)_{n_2} \cdots (q)_{n_{k-1}}} \left[ N_1 \atop n_{k+1} \right]_q,
\end{align*}
where $\left[ m \atop n \right]_q := \frac{(q)_m}{(q)_n(q)_{m-n}}$ are the $q$-binomial coefficients.
For both multi-sums above, $N_j=n_{j}+n_{j+1}+\dotsb+n_k$, in particular
$N_j$ does not end with $n_{k+1}$ in the latter sum.
As before, saddle point analysis gives the system of equations: 
\begin{align*}
\label{TBA-S} (1-Q_k)^2 & =\prod_{j = 1}^{k} Q_{j}^{B_{k,j}}  Q_{k+1}^{-1},\\
 (1-Q_i) &=\prod_{j = 1}^k Q_{j}^{B_{i,j}} Q_{k+1}^{-1}, \ \ 1 \leq i \leq k-1,  \\ 
 (1-Q_{k+1}) & =Q_1^{-1} \cdots Q_k^{-1} Q_{k+1}.
\end{align*}

As argued in \cite[Section 5]{KKMM}, specifically their formula (5.10), we get
$$2 (L(1)-L(Q_{k})) + \left(L(1)- L(Q_{k+1})\right)
+\sum_{i=1}^{k-1} (L(1)-L(Q_i))-\left(L(1)-L\left(\frac12\right)\right) =\frac{\pi^2}{4},$$
with the extra term $L(1)-L(\frac12)=L(\frac12)=\frac{\pi^2}{12}$ coming from the factor $(-q^{1/2})_{N_1}$ in the sum.
It is easy to check that 
$$Q_k=\frac{2k-1}{2k+1}, \  Q_{r}=1-\frac{4}{(2r+1)^2}  \ \ (1 \leq r \leq k-1), \ \ Q_{k+1}=\frac{1}{4}$$
is a solution of the above system. The proof then follows.
Alternatively, we can prove (\ref{super-kirillov}) directly by induction on $k$ using relation (\ref{Rog-help}).
\end{proof}
\begin{remark} 
We note that taking $k \to +\infty$ in (\ref{super-kirillov}) leads to 
\begin{equation} \label{super-kirillov-inf}
\sum_{r \geq 1} L\left(\frac{4}{(2r+1)^2}\right)+L\left(\frac{3}{4}\right)-L\left(\frac12\right)=\frac{\pi^2}{4}.
\end{equation}
This identity can be viewed as a limiting case of Melzer's dilogarithm identity  coming from $N=1$ superconformal characters \cite[Formula (3.12)]{Melzer}:
\begin{equation} \label{melzer}
\sum_{r=1}^{k-1} L \left(\frac{\sin^2 \frac{\pi}{2k}}{\sin^2 \frac{(2r+1) \pi}{4k}}\right)+L\left(1-\frac{1}{4 \cos^2 \frac{\pi}{4k} }\right)-L\left(\frac12 \right)=\pi^2 \frac{(k-1)}{4 k}.
\end{equation}
Indeed, formally taking $ k \to +\infty$ in (\ref{melzer}), we get (\ref{super-kirillov-inf}) by L'hopital's rule.
\end{remark}

\subsection{Further Dilogarithm Identities}

We are not aware of any $q$-series identities connecting 
\begin{equation} \label{new}
\sum_{n_1,n_2,\dotsc,n_k \geq 0}  \frac{q^{\frac{N_1^2}{2}+N_2^2 + N_3^2+ \cdots + N_{k}^2+N_1+N_2+\cdots + N_{k}}}
{(q)_{n_k}^2 (q)_{n_1}(q)_{n_2} \cdots (q)_{n_{k-1}}} 
\end{equation}
with a false theta-type $q$-series. 
The proposition below gives a dilogarithm identity that suggests a relation with a false theta series
is plausible.
\begin{proposition}\label{PropositionMoreDilogarithms2}
For $k\in\N$, let $\{Q_i \}$ be the unique solution of 
\begin{equation} \label{TBA-3}
(1-Q_k)^2= \prod_{1 \leq j \leq k} Q_j^{C_{k,j}}, \ \ (1-Q_i)=\prod_{1 \leq j \leq k} Q_j^{C_{i,j}}; \ \ \  1 \leq i \leq k-1,
\end{equation}
inside the interval $(0,1)$,  where $(C_{i,j})=(2{\rm min}(i,j)-1)$ is the matrix corresponding to the quadratic form in the exponent of the $q$-hypergeometric series (\ref{new}).
Then 
$$ 2(L(1)-L(Q_k))+ \sum_{i = 1}^{k-1}  (L(1)-L(Q_i))=\frac{\pi^2}{5}.$$
\end{proposition}

\begin{proof}
Let $$a=\frac{1}{2}(\sqrt{5}-1).$$
It is not hard to see that 
$$Q_1=a, Q_2=1-\frac{1}{(a+2)^2},....,Q_{k-1},=1-\frac{1}{(a+k-1)^2}, Q_k=1-\frac{1}{a+k}$$
is the unique solution of (\ref{TBA-3}) inside $(0,1)$; when $k=1$ we have $Q_1=1-\frac{1}{a+1}$. 
The rest follows by induction on $k$ using the five-term relation as in Proposition \ref{dilog-log}.
\end{proof}

\section{Final comments on future work}

In a sequel to this paper \cite{JM}, we plan to study $q$-series identities involving expressions with even ``higher poles'' (cf. Section 5)
$$\sum_{n_1,n_2,\dotsc,n_k \geq 0} \frac{q^{\sum_{i=1}^{k}n_i+\sum_{i=1}^{k-1} n_i n_{i+1}}}
{(q)_{n_1}^{r_1}(q)_{n_2}^{r_2} \cdots (q)_{n_k}^{r_k}}.$$
We will  also extend Theorems \ref{euler-false-p} and Theorem \ref{euler-modular-p}
to half-characteristics, where the relevant $q$-hypergeometric series are
$$\sum_{n_1,n_2,\dotsc,n_{k} \geq 0 } 
\frac{q^{\sum_{i=1}^{k} n_i^2-\sum_{i=1}^{k-1} n_{i}n_{i+1}}(-q^{\epsilon})_{n_1}}
{(q)_{n_1}^2 (q)_n^2 \cdots (q)_{n_{k}}^2}.$$
Moreover, motivated by \cite{GPPV}, we will study identities for ``inverted'' $q$-hypergeometric terms, after 
the formal inversion $q \to q^{-1}$.


\begin{thebibliography}{99}

\bibitem{AM2}
D.~Adamovi\'{c} and A.~Milas.
\newblock \emph{The {$N=1$} triplet vertex operator superalgebras: twisted
  sector}.
\newblock SIGMA Symmetry Integrability Geom. Methods Appl., 4: (2008) pp. Paper
  087, 24.

\bibitem{AM1}
D.~Adamovi\'{c} and A.~Milas.
\newblock \emph{The {$N=1$} triplet vertex operator superalgebras}.
\newblock Comm. Math. Phys., 288(1): (2009) pp. 225--270.

\bibitem{AMS} 
D.~Adamovi\'c, A.~Milas, and S.~Sidoli.
\newblock  \emph{The {$N=1$} super singlet vertex operator superalgebra}.
\newblock to appear.

\bibitem{AgarwalAndrewsBressoud1}
A.~K. Agarwal, G.~E. Andrews, and D.~M. Bressoud.
\newblock \emph{The {B}ailey lattice}.
\newblock J. Indian Math. Soc. (N.S.), 51: (1987) pp. 57--73 (1988).

\bibitem{A}
G.~E. Andrews.
\newblock \emph{Partitions and {D}urfee dissection}.
\newblock Amer. J. Math., 101(3): (1979) pp. 735--742.

\bibitem{A2}
G.~E. Andrews.
\newblock \emph{Hecke modular forms and the {K}ac-{P}eterson identities}.
\newblock Trans. Amer. Math. Soc., 283(2): (1984) pp. 451--458.

\bibitem{Andrews1}
G.~E. Andrews.
\newblock \emph{{$q$}-{S}eries: their development and application in analysis,
  number theory, combinatorics, physics, and computer algebra}, vol.~66 of
  \emph{CBMS Regional Conference Series in Mathematics}.
\newblock Published for the Conference Board of the Mathematical Sciences,
  Washington, DC; by the American Mathematical Society, Providence, RI (1986).

\bibitem{Andrews2} G. Andrews,   \emph{The theory of partitions},
Cambridge University Press, Cambridge, 1998.


\bibitem{Bailey1}
W.~N. Bailey.
\newblock \emph{Identities of the {R}ogers-{R}amanujan type}.
\newblock Proc. London Math. Soc. (2), 50: (1948) pp. 1--10.

\bibitem{BeirneOsburn1}
P.~Beirne and R.~Osburn.
\newblock \emph{{$q$}-{S}eries and tails of colored {J}ones polynomials}.
\newblock Indag. Math. (N.S.), 28(1): (2017) pp. 247--260.


\bibitem{Berndt}
B.~C. Berndt.
\newblock \emph{Ramanujan's notebooks. {P}art {III}}.
\newblock Springer-Verlag, New York (1991).

\bibitem{BFOR}
K.~Bringmann, A.~Folsom, K.~Ono, and L.~Rolen.
\newblock \emph{Harmonic {M}aass forms and mock modular forms: theory and
  applications}, vol.~64 of \emph{American Mathematical Society Colloquium
  Publications}.
\newblock American Mathematical Society, Providence, RI (2017).

\bibitem{BHL}
K.~Bringmann, K.~Hikami, and J.~Lovejoy.
\newblock \emph{On the modularity of the unified {WRT} invariants of certain
  {S}eifert manifolds}.
\newblock Adv. in Appl. Math., 46(1-4): (2011) pp. 86--93.

\bibitem{BKM}
K.~Bringmann, J.~Kaszian, and A.~Milas.
\newblock \emph{Higher depth quantum modular forms, multiple {E}ichler
  integrals, and {$\mathfrak{sl}_3$} false theta functions}.
\newblock Res. Math. Sci., 6(2): (2019) pp. Paper No. 20, 41.

\bibitem{BMM}
K.~Bringmann, K.~Mahlburg, and A.~Milas.
\newblock \emph{Quantum modular forms and plumbing graphs of 3-manifolds}.
\newblock J. Combin. Theory Ser. A, 170: (2020) pp. 105145, 32.

\bibitem{BM0}
K.~Bringmann and A.~Milas.
\newblock \emph{{W}-algebras, false theta functions and quantum modular forms,
  {I}}.
\newblock Int. Math. Res. Not. IMRN, (21): (2015) pp. 11351--11387.

\bibitem{BM}
K.~Bringmann and A.~Milas.
\newblock \emph{{W}-algebras, higher rank false theta functions, and quantum
  dimensions}.
\newblock Selecta Math. (N.S.), 23(2): (2017) pp. 1249--1278.

\bibitem{CCFGH}
M.~C.~N. Cheng, S.~Chun, F.~Ferrari, S.~Gukov, and S.~M. Harrison.
\newblock \emph{3d modularity}.
\newblock J. High Energy Phys., (10): (2019) pp. 010, front matter+94.

\bibitem{CS}
C.~C\'{o}rdova and S.-H. Shao.
\newblock \emph{Schur indices, {BPS} particles, and {A}rgyres-{D}ouglas
  theories}.
\newblock J. High Energy Phys., (1): (2016) pp. 040, front matter+37.

\bibitem{CM}
T.~Creutzig and A.~Milas.
\newblock \emph{Higher rank partial and false theta functions and
  representation theory}.
\newblock Adv. Math., 314: (2017) pp. 203--227.

\bibitem{FK}
L.~D. Faddeev and R.~M. Kashaev.
\newblock \emph{Quantum dilogarithm}.
\newblock Modern Phys. Lett. A, 9(5): (1994) pp. 427--434.

\bibitem{FF}
B.~Feigin, E.~Feigin, and I.~Tipunin.
\newblock \emph{Fermionic formulas for characters of {$(1,p)$} logarithmic
  model in {$\Phi_{2,1}$} quasiparticle realisation}.
\newblock In \emph{Exploring new structures and natural constructions in
  mathematical physics}, vol.~61 of \emph{Adv. Stud. Pure Math.}, pp. 161--184.
  Math. Soc. Japan, Tokyo (2011).

\bibitem{FOR}
A.~Folsom, K.~Ono, and R.~C. Rhoades.
\newblock \emph{Mock theta functions and quantum modular forms}.
\newblock Forum Math. Pi, 1: (2013) pp. e2, 27.

\bibitem{GL}
S.~Garoufalidis and T.~T.~Q. L\^{e}.
\newblock \emph{Nahm sums, stability and the colored {J}ones polynomial}.
\newblock Res. Math. Sci., 2: (2015) pp. Art. 1, 55.

\bibitem{GasperRahman}
G.~Gasper and M.~Rahman.
\newblock \emph{Basic hypergeometric series}, vol.~96 of \emph{Encyclopedia of
  Mathematics and its Applications}.
\newblock Cambridge University Press, Cambridge, second ed. (2004).
\newblock With a foreword by Richard Askey.

\bibitem{GPPV}
S.~{Gukov}, D.~{Pei}, P.~{Putrov}, and C.~{Vafa}.
\newblock \emph{{BPS spectra and 3-manifold invariants}}.
\newblock arXiv e-prints, arXiv:1701.06567.

\bibitem{Hajij}
M.~Hajij.
\newblock \emph{The tail of a quantum spin network}.
\newblock Ramanujan J., 40(1): (2016) pp. 135--176.

\bibitem{Hi1}
K.~Hikami.
\newblock \emph{Quantum invariant, modular form, and lattice points}.
\newblock Int. Math. Res. Not., (3): (2005) pp. 121--154.

\bibitem{Hi2}
K.~Hikami.
\newblock \emph{On the quantum invariants for the spherical {S}eifert
  manifolds}.
\newblock Comm. Math. Phys., 268(2): (2006) pp. 285--319.

\bibitem{HL}
K.~Hikami and J.~Lovejoy.
\newblock \emph{Torus knots and quantum modular forms}.
\newblock Res. Math. Sci., 2: (2015) pp. Art. 2, 15.

\bibitem{JM} 
C.~Jennings-Shaffer and A.~Milas.
\newblock \emph{Further q-series identities and conjectures relating false theta functions and characters}.
\newblock preprint, to appear in Contemporary Mathematics.

\bibitem{KKMM}
R.~Kedem, T.~R. Klassen, B.~M. McCoy, and E.~Melzer.
\newblock \emph{Fermionic sum representations for conformal field theory
  characters}.
\newblock Phys. Lett. B, 307(1-2): (1993) pp. 68--76.

\bibitem{KO}
A.~Keilthy and R.~Osburn.
\newblock \emph{Rogers-{R}amanujan type identities for alternating knots}.
\newblock J. Number Theory, 161: (2016) pp. 255--280.

\bibitem{Kirillov}
A.~N. Kirillov.
\newblock \emph{Dilogarithm identities}.
\newblock 118, pp. 61--142 (1995).
\newblock Quantum field theory, integrable models and beyond (Kyoto, 1994).

\bibitem{LZ}
R.~Lawrence and D.~Zagier.
\newblock \emph{Modular forms and quantum invariants of {$3$}-manifolds}.
\newblock pp. 93--107 (1999).
\newblock Sir Michael Atiyah: a great mathematician of the twentieth century.


\bibitem{Melzer}
E.~{Melzer}.
\newblock \emph{{Supersymmetric analogs of the Gordon-Andrews identities, and
  related tba systems}}.
\newblock arXiv e-prints, hep-th/9412154.

\bibitem{Nahm1}
W.~Nahm.
\newblock \emph{Conformal field theory and torsion elements of the {B}loch
  group}.
\newblock In \emph{Frontiers in number theory, physics, and geometry. {II}},
  pp. 67--132. Springer, Berlin (2007).

\bibitem{Nahm2}
W.~Nahm, A.~Recknagel, and M.~Terhoeven.
\newblock \emph{Dilogarithm identities in conformal field theory}.
\newblock Modern Phys. Lett. A, 8(19): (1993) pp. 1835--1847.

\bibitem{Sidoli} S. Sidoli.
\newblock \emph{The $N=1$  Singlet Vertex Superalgebra $\overline{SM(1)}$, }
\newblock PhD thesis, SUNY-Albany: (2018)

\bibitem{Slater1}
L.~J. Slater.
\newblock \emph{A new proof of {R}ogers's transformations of infinite series}.
\newblock Proc. London Math. Soc. (2), 53: (1951) pp. 460--475.

\bibitem{TanneryMolk1}
J.~Tannery and J.~Molk
\newblock \emph{\'{E}l\'{e}ments de la th\'{e}orie des fonctions elliptiques. 
	{T}ome {III}:{C}alcul int\'{e}gral. {P}remi\`ere partie. 
	{T}ome {IV}: {C}alcul int\'{e}gral. {D}euxi\`eme partie},
\newblock Chelsea Publishing Co.: (1972)

\bibitem{VZ}
M.~Vlasenko and S.~Zwegers.
\newblock \emph{Nahm's conjecture: asymptotic computations and
  counterexamples}.
\newblock Commun. Number Theory Phys., 5(3): (2011) pp. 617--642.

\bibitem{Warnaar2}
S.~O. Warnaar.
\newblock \emph{Partial theta functions. {I}. {B}eyond the lost notebook}.
\newblock Proc. London Math. Soc. (3), 87(2): (2003) pp. 363--395.

\bibitem{Warnaar1}
S.~O. Warnaar.
\newblock \emph{Proof of the {F}lohr-{G}rabow-{K}oehn conjectures for
  characters of logarithmic conformal field theory}.
\newblock J. Phys. A, 40(40): (2007) pp. 12243--12254.

\bibitem{Warnaar3}  
S.~O. Warnaar.
\newblock \emph{Partial theta functions}. 
\newblock Encyclopedia Mathematica: preprint.

\bibitem{Warnaar4} 
S.~O. Warnaar.
\newblock {personal communications}.

\bibitem{Zagier1}
D.~Zagier,
\newblock \emph{The {M}ellin transform and related analytic techniques}. 
\newblock appendix to 
\emph{E. Zeidler, Quantum Field Theory I: Basics in Mathematics and Physics. 
	A Bridge Between Mathematicians and Physicists},
pp. 305--323. Springer-Verlag (2006).


\bibitem{Zagier}
D.~Zagier.
\newblock \emph{The dilogarithm function}.
\newblock In \emph{Frontiers in number theory, physics, and geometry. {II}},
  pp. 3--65. Springer, Berlin (2007).
\end{thebibliography}
\end{document}